\documentclass[11pt]{amsart}
\usepackage{url,enumitem, amsopn}
\usepackage{xcolor}
\usepackage[T1]{fontenc}
\usepackage{cleveref}
\usepackage{tikz}
\usepackage{wrapfig}
\usepackage{lipsum}
\usepackage{pgfplots}
\usepackage{tikz}
\usetikzlibrary{patterns}
\pgfdeclarelayer{nodelayer}
\pgfdeclarelayer{edgelayer}
\pgfsetlayers{nodelayer,main,edgelayer}

\newcommand{\cal}{\mathcal}

\newcommand{\reals}{\mbox{$\mathbb R$}}

\newcommand{\nats}{\mbox{$\mathbb N$}}

\newcommand{\floor}[1]{\lfloor #1 \rfloor}
\newcommand{\ceil}[1]{\lceil #1 \rceil}

\newcommand{\comment}[1]{}

\usepackage{comment}

\usepackage{amsmath, amsfonts, amssymb, latexsym, stmaryrd}
\usepackage{graphicx}                   

\def\squarebox#1{\hbox to #1{\hfill\vbox to #1{\vfill}}}
\def\qed{\hspace*{\fill}
        \vbox{\hrule\hbox{\vrule\squarebox{.667em}\vrule}\hrule}\smallskip}

\theoremstyle{plain}
\newtheorem{lemma}{Lemma}[section]
\newtheorem{theorem}[lemma]{Theorem}
\newtheorem{corollary}[lemma]{Corollary}
\newtheorem{proposition}[lemma]{Proposition}

\theoremstyle{definition}

\newtheorem{observation}[lemma]{Observation}
\newtheorem{definition}[lemma]{Definition}

\newtheorem{example}[lemma]{Example}

\newtheorem*{rmk*}{Remark}
\newtheorem*{rmks*}{Remarks}
\newtheorem*{conventions*}{Conventions}
\newtheorem*{convention*}{Convention}

\def\squareforqed{\hbox{\rlap{$\sqcap$}$\sqcup$}}
\def\qed{\ifmmode\squareforqed\else{\unskip\nobreak\hfil
\penalty50\hskip1em\null\nobreak\hfil\squareforqed
\parfillskip=0pt\finalhyphendemerits=0\endgraf}\fi}

\newlength{\tablength}
\newlength{\spacelength}
\newcommand{\tabstar}{\hspace*{\tablength}}
\newcommand{\spacestar}{\hspace*{\spacelength}}
\def\obeytabs{\catcode`\^^I=\active}
{\obeytabs\global\let^^I=\tabstar}
{\obeyspaces\global\let =\spacestar}
\newenvironment{display}{\begingroup\obeylines\obeyspaces\obeytabs}{\endgroup}
\newenvironment{prog}{\begin{display}\parskip0pt\sf}{\end{display}}

\author{Geir Agnarsson}
\address{Department of Mathematical Sciences \\ 
George Mason University \\ Fairfax, VA  22030}
\email{geir@math.gmu.edu}

\author{Miko\l{}aj Sier\.z\k{e}ga}
\address{Department of Mathematical Sciences\\
George Mason University\\
Fairfax, VA 22030}
\address{
Faculty of Mathematics, Informatics and Mechanics\\
University of Warsaw\\
Banacha 2, 02-097 Warsaw, Poland}
\email{msierzeg@gmu.edu\\ m.sierzega@uw.edu.pl}

\title{Elements represented as intersections of sets}



\date{\today}

\begin{document}

\begin{abstract}
  For a natural number $n$ let $[n] = \{1,\ldots,n\}$. We say
  that a family ${\cal{S}}\subseteq 2^{[n]}$ is \emph{representing}
  if every singleton set of $[n]$ is an intersection of some sets
  from ${\cal{S}}$. We show that the smallest possible cardinality
  of a representing set for $[n]$ is the discrete inverse $s(n)$
  of the Sperner's function $n\mapsto \binom{n}{\lfloor n/2\rfloor}$,
  which by Sperner's Theorem is the maximum number of elements in an
  antichain in $2^{[n]}$ when viewed as subset (or boolean) lattice.
  Specifically, $s(n)$ is then the smallest positive integer such that
  $2^{[n]}$ contains an $n$-element antichain. 
  Some generalization, further applications and asymptotics in terms of the 
  second real branch of the Lambert $W$ function are presented.
  \end{abstract} 

\maketitle

\section{Introduction}
\label{sec:intro}

For a natural number $n$ let $[n] = \{1,\ldots,n\}$. In this article
we investigate properties of a family ${\cal{S}}$ of subsets of $[n]$,
so ${\cal{S}}\subseteq 2^{[n]}$, such that each element of $[n]$ can
be represented by an intersection of sets from ${\cal{S}}$. This relates
to identify each element in $[n]$ by distinct labels, one label for each of the
sets in ${\cal{S}}$ it is contained in. As an example, classical binary search 
can be identified by a $0/1$ or ``left/right'' or ``yes/no'' label at each step 
(see Example~\ref{exa:kcube} here below), thereby identifying
an object via a sequence of questions having ``yes/no'' answers~\cite{CT}. In this
case, using a total of $k$ steps amounts to a total of $2k$ types of labels, two
for each step. Needless to say, binary searches, as well as other types of searches,
can be presented as games, where one wants to get as accurate estimate as
possible, using limited ``yes/no'' questions, where a given object is located.
For a comprehensive survey of such search games, see~\cite{sg-survey}. Casually
speaking, the ``yes/no'' answer in each of the $k$ steps in classical binary
search can be viewed as independent of each other.

By not insisting on independence, however, it is possible to perform considerably
fewer searches and this is the purpose of this article, namely to identify
the exact smallest possible number ${\rho}(n)$ of subsets of $[n]$ such 
that each element
of $[n]$ is uniquely determined by being the sole element contained in
an intersection of some of the mentioned subsets of $[n]$
(see Definition~\ref{def:representing} here below). The study
of subsets of a given set, like $[n]$, and their intersections has itself a
long history, but traditionally it concerns itself with a given collection of
sets and the intersections of pairs of the sets of $[n]$ that have a fixed
cardinality, for example in the study of block designs. A classic
survey of such study can be found in~\cite{RyserHJ}.

The remainder of this article is organized as follows: In
Section~\ref{sec:representing} we define our main objects and prove
Theorem~\ref{thm:intrep}, the main result of that section. We also provide
some examples to illuminate the main ideas of the section.

In Section~\ref{sec:set-partitions} we apply and generalize 
Theorem~\ref{thm:intrep} to
certain set partitions. Namely, we demonstrate how we can achieve the
best accuracy in determining the location of a given element when
we don't have optimal number of subsets of $[n]$ in Theorem~\ref{thm:seta}. 
We also present two explicit examples.

Finally, in Section~\ref{sec:formula} we present an asymptotic formula
of the number ${\rho}(n)$ using the lesser known second real branch of the 
Lambert $W$ function~\cite{Lambert-wiki,Lambert-Wolfram} and discuss some 
consequences.

\section{A result on representing family of subsets}
\label{sec:representing}

\paragraph{Conventions} 
Let ${\nats} = \{1,2,\ldots\}$ denote the set of natural numbers,
${\nats}_0 = \{0,1,2,\ldots\}$ the set of nonnegative integers, ${\reals}$ 
the set of real numbers and ${\reals}_+ = \{x\in {\reals} : x > 0\}$ 
the set of positive real numbers. For real numbers $a<b$ we will use the
following notation for real intervals:
\begin{eqnarray*}
[a,b]  =  \{x\in\reals : a\leq x\leq b\}, & \ \ & 
[a,b[  =  \{x\in\reals : a\leq x < b\}, \\
 ]a,b]  =  \{x\in\reals : a < x\leq b\}, & \ \ &
 ]a,b[  =  \{x\in\reals : a < x < b\}. 
\end{eqnarray*}
For $n\in\nats$ we let $[n] = \{1,\ldots,n\}$ 
(as mentioned earlier in the introduction) and we denote the set of all subsets of 
$[n]$ by $2^{[n]}$. Note that $2^{[n]}$ is a boolean lattice w.r.t.~set inclusion and
hence in particular a partially ordered set (poset). Any subset 
${\cal{S}}\subseteq 2^{[n]}$ will then automatically become a sublattice, or a subposet,
of $2^{[n]}$ by inclusion. For $\ell\leq n$ denote by 
$\binom{[n]}{\ell}\subseteq 2^{[n]}$ the set of all subsets of $[n]$ that consist of
$\ell$ elements. 
A \emph{partition} of a finite set $X$ is collection of subsets 
$\{P_1,\ldots,P_q\}$ of $X$ that are (i) pairwise disjoint, so $P_i\cap P_j = \emptyset$
when $i\neq j$ and that (ii) cover $X$, so 
$P_1\cup\cdots\cup P_q = X$\footnote{Strictly speaking, this is what many call
a \emph{weak partition} since we do not insist here 
that $P_i\neq\emptyset$ for each $i$.}.
For real functions $f, g : {\reals}_+\rightarrow{\reals}_+$,
(resp.~integer functions $f,g: {\nats}\rightarrow {\nats}$), we let $f(x)\sim g(x)$ as $x$ tends to infinity denote that $\lim_{x\rightarrow\infty}\frac{f(x)}{g(x)} = 1$
(resp. we let $f(n)\sim g(n)$ as $n$ tends to infinity as an integer denote that 
$\lim_{n\rightarrow\infty}\frac{f(n)}{g(n)} = 1$). 
The natural logarithm (with base number $e\approx 2.718281828...$) will be denoted by $\log$ and the base-2 logarithm by $\lg$.
\begin{definition}
  \label{def:representing}
  A family of subsets ${\cal{S}}\subseteq 2^{[n]}$ \emph{represents} 
  $[n]$, or is \emph{representing} for $[n]$,  if for every $i\in [n]$ the singleton 
  $\{i\}$ is the intersection of some sets from  $\cal S$. For each $n\in\nats$ we let 
  ${\rho}(n):=\min\Big\{|\cal S|\,:\cal{S} \mbox{ represents }[n]\Big\}$.
\end{definition}
\begin{example}
Trivially, any family $\mathcal{S}\subseteq 2^{[n]}$ that contains all singletons
$\{1\},\{2\},\ldots,\{n\}$ represents $[n]$. In particular $\cal E(n)\leq n$.
\end{example}
\begin{example}
  \label{exa:mn} 
  Let $m,n\in\nats$. By relabeling the elements in a classical 
  $[m]\times[n]$ array by the integers $1,2,\ldots, mn$ in an any fashion, i.e., arranging numbers $1,2,\dots, mn$ in an 
  array with $m$ rows and $n$ columns,  we see that
  ${\cal{S}} = \left(\bigcup_{i\in[m]}\{i\}\times[n]\right) 
  \cup\left(\bigcup_{j\in[n]}[m]\times\{j\}\right)$
  represents $[m]\times[n]$ since every singleton is identified by the intersection of 
  the relevant row and column: $\{(a,b)\} = \{a\}\times[n]\cap[m]\times\{b\}$, see
  Figures~\ref{fig:RRCCC}, \ref{fig:3x2matrix} for an illustration of this for $m=3$ and
  $n=2$. Here we have $|\cal S|= m+n$, so if $m,n>1$ and $m>2$ or $n>2$ then 
  $\rho(mn) \leq |\cal{S}| = m+n < mn$ showing that $\rho(n) < n$ for each 
  composite integer $n\geq 6$.
\begin{figure}[htbp!]
\centering
\begin{tikzpicture}[scale=1]

\fill[pattern=north east lines, pattern color=gray] (0,2) rectangle (3,3);
\fill[pattern=north west lines, pattern color=gray] (8,0.5) rectangle (9,2.5);

\draw (0,2) grid (3,3);
\draw (0,0) grid (3,1);

\node at (0.5,2.5) {1};
\node at (1.5,2.5) {2};
\node at (2.5,2.5) {3};

\node at (0.5,0.5) {4};
\node at (1.5,0.5) {5};
\node at (2.5,0.5) {6};

\node[left] at (0,2.5) {$R_1:$};
\node[left] at (0,0.5) {$R_2:$};

\node[left] at (5,1.5) {$C_1:$};
\draw (5,0.5) rectangle (6,2.5);
\draw (5,1.5) -- (6,1.5);

\node at (5.5,2) {1};
\node at (5.5,1) {4};

\node[left] at (8,1.5) {$C_2:$};
\draw (8,0.5) rectangle (9,2.5);
\draw (8,1.5) -- (9,1.5);

\node at (8.5,2) {2};
\node at (8.5,1) {5};

\node[left] at (11,1.5) {$C_3:$};
\draw (11,0.5) rectangle (12,2.5);
\draw (11,1.5) -- (12,1.5);

\node at (11.5,2) {3};
\node at (11.5,1) {6};

\end{tikzpicture}

\caption{The representing family $\cal S=\{R_1,R_2,C_1,C_2,C_3\}$.}
\label{fig:RRCCC}
\end{figure}
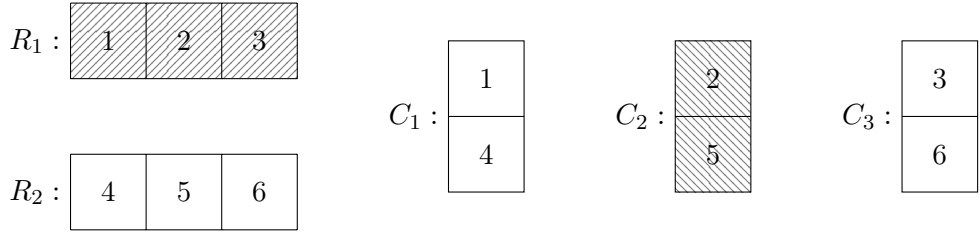
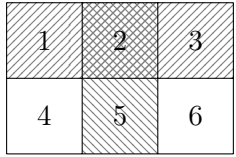
\begin{figure}[htbp!]
\centering
\begin{tikzpicture}[scale=1]


\fill[pattern=north east lines, pattern color=gray] (0,1) rectangle (3,2);

\fill[pattern=north west lines, pattern color=gray] (1,0) rectangle (2,2);

\fill[pattern=north east lines, pattern color=gray] (1,1) rectangle (2,2);

\draw[step=1,black] (0,0) grid (3,2);

\node at (0.5,1.5) {1};
\node at (1.5,1.5) {2};
\node at (2.5,1.5) {3};

\node at (0.5,0.5) {4};
\node at (1.5,0.5) {5};
\node at (2.5,0.5) {6};

\end{tikzpicture}
\caption{The set $[6]$ identified with a $3\times 2$ array. The singleton $\{2\}$ is represented as an intersection of $R_1$ and $C_2$.}
\label{fig:3x2matrix}
\end{figure}
\end{example}

Let $m\leq n$ be natural numbers and suppose ${\cal{S}}\subseteq 2^{[n]}$
represents $[n]$. If $i\in [m]\subseteq [n]$, then there are sets
$A_1,\ldots,A_k\in {\cal{S}}$ such that $\{i\} = A_1\cap\cdots\cap A_k$
and hence $\{i\} = B_1\cap\cdots\cap B_k$ where $B_i = A_i\cap[m]$.
Therefore ${\cal{S}} = \{A\cap[m] : A\in {\cal{S}}\}$ represents $[m]$
and we have the following.
\begin{observation}
  \label{obs:monoton}
  If $m\leq n$ are natural numbers and $[n]$ can be represented by a 
  family containing $k$ subsets of $[n]$, then $[m]$ can be represented
  by at most $k$ subsets of $[m]$. Consequentially, ${\rho}(m)\leq {\rho}(n)$
  and so ${\rho} : \nats \rightarrow \nats$ is an increasing function. 
\end{observation}
\begin{example}
\label{exa:primes}
Amusingly, if $p\geq 7$ is a prime, then we have by Example~\ref{exa:mn} and
Observation~\ref{obs:monoton} that 
\[
{\rho}(p)\leq {\rho}(p+1) = {\rho}\left(2\cdot\frac{p+1}{2}\right)
\leq 2 + \frac{p+1}{2} < p
\]
and so ${\rho}(n) < n$ for every natural number $n\geq 6$, composite or prime.
\end{example}
\begin{example}
  \label{exa:kcube}
  Expanding on the above Examples~\ref{exa:mn} and \ref{exa:primes} and
  Observation~\ref{obs:monoton}, for each $k\in{\nats}_0$
  we can label the elements of $[2^k]$ by $\{0,1\}^k$. For each
  $i\in [k]$ let $A_i$, resp.~$B_i$, denote the set of elements of
  $\{0,1\}^k$ with $i$-th coordinate equal to $0$, resp.~$1$.
  In this case the family $\{A_i,B_i : i\in[k]\}$ of $2k$ sets
  represents $\{0,1\}^k$. If
  $n\in\nats$ and $k = \ceil{\lg n}$, then $[n]$ can be
  labeled by a subset of $\{0,1\}^k$ and hence, by
  Observation~\ref{obs:monoton} we therefore see that $[n]$ can be
  represented by $2k = 2\ceil{\lg n}$ subsets of $[n]$.
\end{example}
If the family ${\cal{S}} = \{A_1,\ldots,A_k\}$ represents the set
$[n]$, then we clearly
have an injection $[n]\hookrightarrow 2^{[k]}$ determined by some choice
of a subset $\{j_1,\ldots,j_h\}$ with $\{i\} = A_{j_1}\cap\cdots\cap A_{j_h}$
for each $i\in [n]$ and therefore
$n\leq 2^k - 1$. This together with Example~\ref{exa:kcube} shows that
the smallest possible cardinality of a family representing $[n]$ has
order ${\rho}(n)$ which is then between $\ceil{\lg(n+1)}$ and 
$2\ceil{\lg n}$.
Further, the image of the mentioned injection must be an antichain of
$2^{[k]}$ when ordered by inclusion, since if $i,j\in [n]$ are distinct
and $\{i\} = \bigcap_{h\in I}A_h$ and $\{j\} = \bigcap_{h\in J}A_h$, then
$I$ and $J$ must be incomparable w.r.t.~inclusion. By 
Sperner's Theorem~\cite{Sperner-wiki,Sperner}
the maximal cardinality of an antichain in $2^{[k]}$ is given by
$\binom{k}{\lfloor k/2\rfloor}$ and so we additionally have 
$\binom{k}{\lfloor k/2\rfloor}\geq n$. The lower bound of $k$ in terms
of $n$ implied by this inequality will indeed suffice as we will see. 
\begin{definition}
    \label{def:s(n)}
For $n\in\nats$ let $s(n)$ denote the smallest $k$ such that
$\binom{k}{\lfloor k/2\rfloor}\geq n$, so $s(n)$ is the smallest cardinality
of a finite set that has an antichain of $n$ subsets when ordered by
inclusion.    
\end{definition}
Note that $s(n)$ is the smallest cardinality of a finite set
that has $n$ subsets where no set contains another. 
The integer sequence $(s(n))_{n\geq 1}$ appears in
\emph{The On-Line Encyclopedia Of Integer Sequences}~\cite{oeis}
as sequence A305233~\cite{A305233}; the discrete inverse of the Sperner's
function $n\mapsto \binom{n}{\lfloor n/2\rfloor}$.
     
We can now state the following main result of this section.
\begin{theorem}
  \label{thm:intrep}
  For $n\in\nats$ we have ${\rho}(n) = s(n)$, so the smallest cardinality of a family
  ${\cal{S}}\subseteq 2^{[n]}$ that can represent $[n]$ equals $s(n)$.
\end{theorem}
\begin{proof}
  Let $k\in\nats$ and let ${\left[\binom{k}{\lfloor k/2\rfloor}\right]}$ be labeled
  by the corresponding subsets $\binom{[k]}{\lfloor k/2\rfloor}$ of
  $[k]$ of cardinality $\lfloor k/2\rfloor$. For each $i\in [k]$ let
  $A_i = \{S\in\binom{[k]}{\lfloor k/2\rfloor} : i\in S\}$. That
  ${\cal{S}} = \{A_i : i\in[k]\}\subseteq 2^{[k]}$ represents
  $\binom{[k]}{\lfloor k/2\rfloor}$ is clear from construction: for
  each $S\in\binom{[k]}{\lfloor k/2\rfloor}$ we clearly have
  $\bigcap_{i\in S}A_i = \{S\}$. By Observation~\ref{obs:monoton}
  we have that for any $n \leq \binom{k}{\lfloor k/2\rfloor}$ the set $[n]$
  can be represented by $k$ subsets of $[n]$. Therefore, by the above
  definition of $s(n)$, the set $[n]$ can be represented by $s(n)$ subsets
  of $[n]$ and hence ${\rho}(n)\leq s(n)$.
  
  By the discussion before Definition~\ref{def:s(n)} we see that $s(n)$ is
  indeed the smallest possible integer such that $[n]$ can be represented by $s(n)$
  subsets of $[n]$ and so $s(n)\leq {\rho}(n)$.
  \end{proof}
Slightly more can be said about the intersecting representation. By the
above proof of Theorem~\ref{thm:intrep} we have further the following.
\begin{proposition}
  \label{prp:uniform}
  For $n\in\nats$ there is a family ${\cal{S}}\subseteq 2^{[n]}$
  of cardinality $s(n)$ such that each singleton set of $[n]$ is represented
  as an intersection of exactly $\lfloor s(n)/2\rfloor$ subsets of $[n]$, each subset
  of cardinality $\binom{s(n)-1}{\lfloor s(n)/2\rfloor -1}$.
\end{proposition}
\begin{example}
  \label{exa:n=6}
  For $n=6$ we note $s(6) = 4$ since $\binom{4}{2} = 6$ and $\lfloor s(n)/2\rfloor = 2$.
  By listing the six elements of $\binom{[4]}{2}$ lexicographically
  \[
  \{1,2\} < \{1,3\} < \{1,4\} < \{2,3\} < \{2,4\} < \{3,4\},
  \]
  we can identify each set in $\binom{[4]}{2}$ by its order in this listing.
  This is equivalent to labeling of the edges of the complete graph $K_4$ on the
  vertices $[4]$ by the labels $1,2,\ldots,6$ such that edge $\{1,2\}$ is
  labeled by $1$, the edge $\{1,3\}$ is labeled by $2$ etc.,
  see Figure~\ref{fig:K4}.
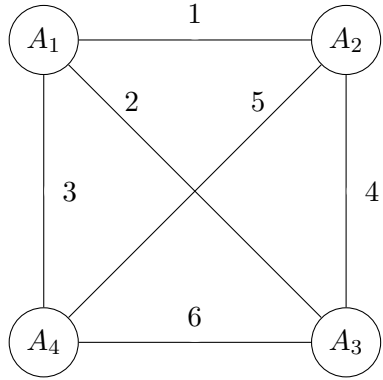
\begin{figure}[h!]
    \centering
    \begin{tikzpicture}[every  node/.style={circle,draw,fill=white}]
  
\node (A) at (-2,2) {$A_1$};
\node (B) at (2,2) {$A_2$};
\node (C) at (2,-2) {$A_3$};
\node (D) at (-2,-2) {$A_4$};

        \draw [-] (A) edge node[circle=none,draw=none,midway, above] {$1$} (B) (A) edge node[circle=none,draw=none,fill=none,pos=0.25, above] {$2$} (C) (A) edge node[circle=none,draw=none,midway, right] {$3$} (D)  (B) edge node[circle=none,draw=none,midway, right] {$4$} (C) (B) edge node[circle=none,draw=none,fill=none,pos=0.25, above] {$5$} (D) (C) edge node[circle=none,draw=none,midway, above] {$6$}(D) ;
    \end{tikzpicture}
    \caption{A graph representation using $K_4$; the complete graph on four vertices.}
    \label{fig:K4}
\end{figure}
In this way we obtain the sets
\[
A_1 = \{1,2,3\}, \ \ A_2 = \{1,4,5\}, \ \ A_3 = \{2,4,6\}, \ \ A_4 = \{3,5,6\},
\]
and the representation
\begin{eqnarray*}
  \{1\} = A_1\cap A_2, & \{2\} = A_1\cap A_3, & \{3\} = A_1\cap A_4, \\
  \{4\} = A_2\cap A_3, & \{5\} = A_2\cap A_4, & \{6\} = A_3\cap A_4, \\
\end{eqnarray*}
where the indices of the sets forming the representation of $\{i\}$ 
indeed make up the $i$-th set of $\binom{[4]}{2}$ in the lexicographical order of them. 
Finally note that each singleton 
set is the intersection of $\floor{s(6)/2} = 2$ sets $A_i$ and each $A_i$
has exactly $3 = \binom{4-1}{2-1}$ elements.
\end{example}

\section{An application to set partitions}
\label{sec:set-partitions}

In this section we apply our results from Section~\ref{sec:representing} 
to investigate properties of a family $\cal{S}$ of subsets of $[n]$ such that
each element of $[n]$ can be "narrowed down"to a subset of a given cardinality,
instead of pinpointed exactly. For example, how many subsets are needed to pinpoint
each element of $[n]$ as one of seven possible elements, or pinpoint each element 
of $[n]$ within a given accuracy of a certain percentage points, say $11$ percent. 

Suppose $k,n\in{\nats}$ are such that $n\geq\binom{k}{\lfloor k/2\rfloor}$ and we
have a partition 
${\cal{P}} = \left\{P_S : S\in \binom{[k]}{\lfloor k/2\rfloor}\right\}$ 
indexed by subsets of $[k]$ of cardinality $\lfloor k/2\rfloor$. Analogous 
to the proof of Theorem~\ref{thm:intrep} we let 
$A_i' = \left\{ P_S : S\in \binom{[k]}{\lfloor k/2\rfloor}, \ i\in S\right\}$
and then $\cal{S}' = \{ A_i' : i\in [k]\}$ represents the set ${\cal{P}}$ and so
$\bigcap_{i\in S}A_i' = \{P_S\}$ for each $S\in \binom{[k]}{\lfloor k/2\rfloor}$. 

Since ${\cal{P}}$ is a partition of $[n]$ we have for each $i\in [k]$ 
that $B_i = \bigcup_{P_S\in A_i'}P_S \subseteq [n]$ and further that
\[
\bigcap_{i\in S}B_i = \bigcap_{i\in S}\left(\bigcup_{P_S\in A_i'}P_S\right) 
= P_S.
\]
For the given $k,n\in{\nats}$ we can clearly find a partition of $[n]$ into 
$\binom{k}{\lfloor k/2\rfloor}$ parts where each part has cardinality at most
$\left\lceil n\Big/\binom{k}{\lfloor k/2\rfloor}\right\rceil$. 
From this we can deduce the following.
\begin{corollary}
\label{cor:set1}
Let $a,k,n\in{\nats}$ be such that $a\leq n$ and $n\geq\binom{k}{\lfloor k/2\rfloor}$.
If  $\left\lceil n\Big/\binom{k}{\lfloor k/2\rfloor}\right\rceil \leq a$, then there 
are $k$ subsets $B_1,\ldots,B_k$ of $[n]$ such that for each element element $i\in [n]$
there is an intersection $P = \bigcap_{j\in I}B_j$ of some of the $B_j$s that contains
$i$ such that $|P|\leq a$.
\end{corollary}
Note that when $a = 1$, the smallest $k$ satisfying the condition in the above
Corollary~\ref{cor:set1} is the smallest $k$ satisfying $n\leq \binom{k}{\lfloor k/2\rfloor}$,
that is to say $k = s(n)$ from Definition~\ref{def:s(n)} and this
is the best (smallest) $k$ possible as stated in the said Theorem~\ref{thm:intrep}. We now 
argue that for given natural numbers $a$ and $n$, the smallest $k$ satisfying the condition 
in the above Corollary~\ref{cor:set1} is also the best possible. If $k$ satisfies the mentioned
condition in Corollary~\ref{cor:set1} we obtained directly the existence of $k$ subsets $B_1,\ldots,B_k$
with the desired properties. That the smallest such $k$ is the best possible takes a tad more work
than for the singleton case when $a=1$.

Suppose $a,k,n\in{\nats}$ are as in Corollary~\ref{cor:set1}, and that $k$ is as 
small as possible. Suppose there are $k-1$ subsets
$B_1,\ldots, B_{k-1}$ of $[n]$ such that each $i\in [n]$ is contained in some intersection 
of the $B_j$s, where the cardinality of each such intersection is at most $a$.
In this case we have a map $i\mapsto P_i$ such that $i\in P_i$ and each 
$P_i$ is an intersection of some of the sets $B_1,\ldots, B_{k-1}$ and $|P_i|\leq a$ for each
$i\in [n]$. If $P_i\subseteq P_j$ for some $i,j\in [n]$ then we can toss out $P_i$ and 
use $P_j$ instead of $P_i$. Hence, we can assume that the distinct elements $P_i$ where $i\in [n]$ 
form an antichain in $2^{[n]}$, that is we can assume that ${\cal{P}} = \{P_i : i\in [n]\}$ is
an antichain. Suppose $P,P'\in {\cal{P}}$ are distinct, so neither $P\setminus P'$ nor
$P'\setminus P$ is empty. Since each of them is an intersection of the $B_j$s, so 
$P = \bigcap_{\ell\in I}B_{\ell}$ and $P' = \bigcap_{\ell\in I'}B_{\ell}$ where $I,I'\subseteq [k-1]$,
then each of the index sets $I$ and $I'$ also are incomparable as subsets of $[k-1]$.
Since each $P\in {\cal{P}}$ is uniquely determined by its corresponding index set $I\subseteq [k-1]$,
the number of elements in ${\cal{P}}$ is at most the maximum number of an antichain in $2^{[k-1]}$
and so $|{\cal{P}}|\leq \binom{k-1}{\lfloor (k-1)/2\rfloor}$. By definition of $k$ we have
$\left\lceil n\Big/\binom{k-1}{\lfloor (k-1)/2\rfloor}\right\rceil > a$ and therefore
\[
n = \left| \bigcup_{P\in{\cal{P}}}P\right| \leq \sum_{P\in {\cal{P}}}|P| 
\leq \binom{k-1}{\lfloor (k-1)/2\rfloor}a, 
\]
contradicting the definition of $k$. We can now state the main theorem of this section.
\begin{theorem}
\label{thm:seta}
Let $a,n\in{\nats}$ be such that $a\leq n$. The smallest $k$ such there are $k$ 
subsets $B_1,\ldots,B_k$ of $[n]$ with the property that each element element $i\in [n]$
is contained in an intersection $P = \bigcap_{j\in I}B_j$ of some of the $B_j$s with
$|P|\leq a$ is equal to the smallest $k\in {\nats}$ such that
$\left\lceil n\Big/\binom{k}{\lfloor k/2\rfloor}\right\rceil \leq a$.
\end{theorem}
By Corollary~\ref{cor:set1} and the above Theorem~\ref{thm:seta} we obtain the
following approximation version.
\begin{corollary}
\label{cor:set2}
Let $n\in{\nats}$ and $\alpha\in ]0,1[$ be a real number. 
The smallest $k$ such there are $k$ 
subsets $B_1,\ldots,B_k$ of $[n]$ with the property that each element element $i\in [n]$
is contained in an intersection $P = \bigcap_{j\in I}B_j$ of some of the $B_j$s with
$|P|/n\leq \alpha$, is equal to the smallest $k\in {\nats}$ such that
$\left\lceil n\Big/\binom{k}{\lfloor k/2\rfloor}\right\rceil \leq \alpha n$.
\end{corollary}
\begin{example}
\label{exa:23and4}
Suppose $n = 21$ and $\alpha = 1/5$ or $20$ percent. We want to determine 
the smallest number $k$ of subsets $B_1,\ldots,B_k$ of $[n] = \{1,2,3,\ldots,21\}$ such
that each integer in $\{1,2,3,\ldots,21\}$ is contained in some intersection 
$P$ of the sets $B_1\ldots,B_k$ where $|P|/21 \leq \frac{1}{5}$.
By the above Corollary~\ref{cor:set2} we obtain for $k=4$ that 
\[
\left\lceil n\Big/\binom{k}{\lfloor k/2\rfloor}\right\rceil = \left\lceil 21\Big/\binom{4}{2}\right\rceil 
= \lceil21/6\rceil = 4 \leq \frac{1}{5}\cdot21 = \alpha n
\]
and $k=4$ is the smallest such natural number as this does not hold for $k=3$. Hence we 
can choose any partition of $[21]$ into $\binom{k}{\lfloor k/2\rfloor} = \binom{4}{2} = 6$ parts
where each part has at most $\lceil 21/6\rceil = 4$ elements. Any such partition will do for our
purposes, for example
\[
\begin{array}{lll}
    P_1 = \{1,2,3,4\}, & P_2 = \{5,6,7,8\}, & P_3 = \{9,10,11,12\} \\
    P_4 = \{13,14,15,16\}, & P_5 = \{17,18,19\}, & P_6 = \{20,21\}.   
\end{array}
\]
Akin to Example~\ref{exa:n=6} we now obtain the sets $B_i$s from this partition 
as the $A_i$ were obtained from the singleton sets
\begin{eqnarray*}
B_1 & = & P_1\cup P_2\cup P_3 = \{1,2,3,4,5,6,7,8,9,10,11,12\}, \\
B_2 & = & P_1\cup P_4\cup P_5 = \{1,2,3,4,13,14,15,16,17,18,19\}, \\
B_3 & = & P_2\cup P_4\cup P_6 = \{5,6,7,8,13,14,15,16,20,21\}, \\
B_4 & = & P_3\cup P_5\cup P_6 = \{9,10,11,12,17,18,19,20,21\}, \\
\end{eqnarray*}
and the representation
\begin{eqnarray*}
  P_1 = B_1\cap B_2, & P_2 = B_1\cap B_3, & P_3= B_1\cap B_4, \\
  P_4 = B_2\cap B_3, & P_5 = B_2\cap B_4, & P_6 = B_3\cap B_4, 
\end{eqnarray*}
see Figure~\ref{fig:K4-part}.
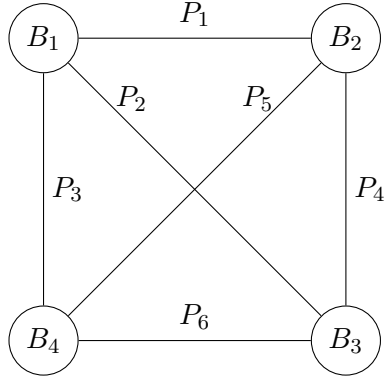
\begin{figure}[h!]
    \centering
    \begin{tikzpicture}[every  node/.style={circle,draw}]
\node (A) at (-2,2) {$B_1$};
\node (B) at (2,2) {$B_2$};
\node (C) at (2,-2) {$B_3$};
\node (D) at (-2,-2) {$B_4$};

        \draw [-] (A) edge node[circle=none,draw=none,midway, inner sep=1pt, above] {$P_1$} (B) (A) edge node[circle=none,draw=none,pos=0.25, inner sep=2pt, above] {$P_2$} (C) (A) edge node[circle=none,draw=none,midway, inner sep=1pt, right] {$P_3$} (D)  (B) edge node[circle=none,draw=none, inner sep=1pt,midway, right] {$P_4$} (C) (B) edge node[circle=none,draw=none, inner sep=2pt,pos=0.25, above] {$P_5$} (D) (C) edge node[circle=none,draw=none, inner sep=1pt,midway, above] {$P_6$}(D) ;
    \end{tikzpicture}
    \caption{An application of the case for $n=6$ using $K_4$ with six edges.}
    \label{fig:K4-part}
\end{figure}
In this way each $i\in \{1,2,3,\ldots,21\}$ is contained in some
intersection of sets $B_1,B_2,B_3,B_4$ where each of these intersections
$P$ has cardinality at most $4 \leq \frac{1}{5}\cdot21=\alpha n$. Therefore
by the use of these four subsets we can pinpoint each element in $\{1,2,3,\ldots,21\}$
with an accuracy of a least $20$ percent. 
\end{example}
\begin{example}
\label{exa:101and5}
For a more elaborate example, let $n = 101$ and $a = 5$. We want to determine 
the smallest number $k$ of subsets $B_1,\ldots,B_k$ of $[n] = \{1,2,3,\ldots,101\}$ such
that each integer in $[n] = [101] = \{1,2,3,\ldots,101\}$ is contained in some intersection 
$P$ of the sets $B_1\ldots,B_k$ where $|P|\leq 5$.
By the above Corollary~\ref{cor:set2} we obtain for $k=7$ that 
\[
\left\lceil n\Big/\binom{k}{\lfloor k/2\rfloor}\right\rceil = \left\lceil 101\Big/\binom{7}{3}\right\rceil 
= \lceil101/35\rceil = 3 \leq 5
\]
and $k=7$ is the smallest such natural number, since for $k=6$ we have
\[
\left\lceil n\Big/\binom{k}{\lfloor k/2\rfloor}\right\rceil = \left\lceil 101\Big/\binom{6}{3}\right\rceil 
= \lceil101/20\rceil = 6 > 5.
\]
So, we can choose any partition of $[101]$ into $\binom{k}{\lfloor k/2\rfloor} = \binom{7}{3} = 35$ parts
where each part has at most $\lceil 101/35\rceil = 3$ elements. Any such partition 
will do for our purposes. For example
\[
P_1 = \{1,2,3\}, \ \ P_2 = \{4,5,6\},\ \ldots, \ \ P_{33} = \{97,98,99\}, \ \ P_{34} = \{100,101\}, \ \ P_{35} = \emptyset.
\]
Unlike the above Example~\ref{exa:23and4} we now have $34$ nonempty 
parts where $34 \neq \binom{k}{\lfloor k/2\rfloor}$ 
for any integer $k$. For $k=7$ we can label all the 
$\binom{k}{\lfloor k/2\rfloor} = \binom{7}{3} = 35$ element sets of $\binom{[7]}{3}$ lexicographically
\begin{equation}
\label{eqn:7choose3-list}
\begin{array}{l}
  \{1,2,3\} < \{1,2,4\} < \{1,2,5\} < \{1,2,6\} < \{1,2,7\} < \{1,3,4\} < \{1,3,5\} < \\ 
  \{1,3,6\} < \{1,3,7\} < \{1,4,5\} < \{1,4,6\} < \{1,4,7\} < \{1,5,6\} < \{1,5,7\} < \\
  \{1,6,7\} < \{2,3,4\} < \{2,3,5\} < \{2,3,6\} < \{2,3,7\} < \{2,4,5\} < \{2,4,6\} < \\
  \{2,4,7\} < \{2,5,6\} < \{2,5,7\} < \{2,6,7\} < \{3,4,5\} < \{3,4,6\} < \{3,4,7\} < \\
  \{3,5,6\} < \{3,5,7\} < \{3,6,7\} < \{4,5,6\} < \{4,5,7\} < \{4,6,7\} < \{5,6,7\}
\end{array}
\end{equation}
and we can identify each set in $\binom{[7]}{3}$ by its order in this listing.
This is equivalent to the labeling of the triangles of the complete graph $K_7$ on the
vertices $\{B_1,\ldots,B_7\}$ by the labels $1,2,\ldots, 35$ such that the
triangle $\{1,2,3\}$ is labeled by $1$, the triangle $\{1,2,4\}$ is labeled by $2$
etc.~and the final triangle $\{5,6,7\}$ is labeled by $35$, see Figure~\ref{fig:K7-part}
where two triangles of $K_7$ are shown: one with vertices $B_1,B_3$ and $B_5$ that has label $7$
(since $\{1,3,5\}$ is the $7$-th $3$-set in $\binom{[7]}{3}$ in the above lexicographical ordering,
in~Figure~\ref{fig:K7-part})
and the last triangle with vertices $B_5,B_6$ and $B_7$ that has label $35$.
\begin{figure}[h!]
    \centering
    \begin{tikzpicture}[every  node/.style={circle,draw,fill=white}]
      \fill[gray!30] (0,4.5) -- (4.387176,-1.0013445) -- (-1.952478,-4.0543605)  -- cycle;
      node/.style={circle,draw,fill=white}]
      \fill[gray!30] (-3.5182395,2.805705) -- (-1.952478,-4.0543605) -- (-4.387176,-1.0013445)  -- cycle;
\node (A) at (0,4.5) {$B_1$};
\node (B) at (3.5182395,2.805705) {$B_2$};
\node (C) at (4.387176,-1.0013445) {$B_3$};
\node (D) at (1.952478,-4.0543605) {$B_4$};
\node (E) at (-1.952478,-4.0543605) {$B_5$};
\node (F) at (-4.387176,-1.0013445) {$B_6$};
\node (G) at (-3.5182395,2.805705) {$B_7$};

\node[circle=none,draw=none,fill=none,font=\Large] (H) at ({1.5*sqrt(7)/6},{1.5*(-cos(pi/7)/3)}) {$P_7$};
\node[circle=none,draw=none,fill=none,font=\Large] (M) at (-3.2859645,-0.75) {$P_{35}$};

        \draw [-] (A) edge (B) (A) edge (C) (A) edge (G) (A) edge (E) (B) edge (C) (C) edge (D) (C) edge (E) (D) edge (E) (E) edge (F) (E) edge (G) (F) edge (G);
    \end{tikzpicture}
    \caption{Each triangle represents a part of the partition.}
    \label{fig:K7-part}
\end{figure}
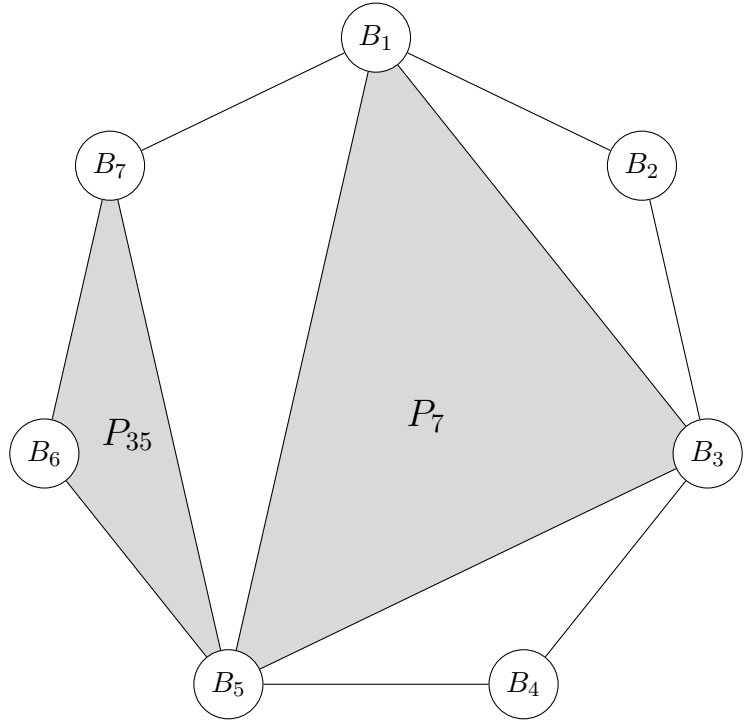
In this way we obtain the sets $B_1,\ldots, B_7$ where each $B_i$ consists of
the union of all the $\binom{7-1}{3-1} = \binom{6}{2} = 15$ sets $P_j$ where each $j$ corresponds
to a set in $\binom{[7]}{3}$ that contains $i$ which we can see from (\ref{eqn:7choose3-list}).
Recall that $P_{35} = \emptyset$ in the following unions. By scanning the list in 
(\ref{eqn:7choose3-list}) we obtain 
\begin{eqnarray*}
B_1 & = & P_1\cup P_2\cup P_3\cup P_4\cup P_5\cup P_6\cup P_7\cup P_8\cup P_9\cup P_{10}
  \cup P_{11}\cup P_{12}\cup P_{13}\cup P_{14}\cup P_{15} \\
B_2 & = & P_1\cup P_2\cup P_3\cup P_4\cup P_5\cup P_{16}\cup P_{17}\cup P_{18}\cup P_{19}\cup P_{20}
  \cup P_{21}\cup P_{22}\cup P_{23}\cup P_{24}\cup P_{25} \\
B_3 & = & P_1\cup P_6\cup P_7\cup P_8\cup P_9\cup P_{16}\cup P_{17}\cup P_{18}\cup P_{19}\cup P_{26}
  \cup P_{27}\cup P_{28}\cup P_{29}\cup P_{30}\cup P_{31} \\
B_4 & = & P_2\cup P_6\cup P_{10}\cup P_{11}\cup P_{12}\cup P_{16}\cup P_{20}\cup P_{21}\cup P_{22}\cup P_{26}
  \cup P_{27}\cup P_{28}\cup P_{32}\cup P_{33}\cup P_{34} \\ 
B_5 & = & P_3\cup P_7\cup P_{10}\cup P_{13}\cup P_{14}\cup P_{17}\cup P_{20}\cup P_{23}\cup P_{24}\cup P_{26}
  \cup P_{29}\cup P_{30}\cup P_{32}\cup P_{33}\cup P_{35} \\
B_6 & = & P_4\cup P_8\cup P_{11}\cup P_{13}\cup P_{15}\cup P_{18}\cup P_{21}\cup P_{23}\cup P_{25}\cup P_{27}
  \cup P_{29}\cup P_{31}\cup P_{32}\cup P_{34}\cup P_{35} \\
B_7& = & P_5\cup P_9\cup P_{12}\cup P_{14}\cup P_{15}\cup P_{19}\cup P_{22}\cup P_{24}\cup P_{25}\cup P_{28}
  \cup P_{30}\cup P_{31}\cup P_{33}\cup P_{34}\cup P_{35}. 
\end{eqnarray*}
Since $[101] = P_1\cup\cdots\cup P_{35}$ is a partition we obtain the representation of each $P_i$ as
the intersection of all the $B_j$ that contain the part $P_i$ as follows. This can be seen from the above 
display and also directly from listing in (\ref{eqn:7choose3-list}).
\[
\begin{array}{llll}
P_1 = B_1 \cap B_2\cap B_3, \ \ & 
P_2 = B_1 \cap B_2 \cap B_4 , \ \ &  
P_3 = B_1 \cap B_2 \cap B_5 , \ \ & 
P_4 = B_1 \cap B_2 \cap B_6 , \\
P_5 = B_1 \cap B_2 \cap B_7 , \ \ &
P_6 = B_1 \cap B_3 \cap B_4 , \ \ &
P_7 = B_1 \cap B_3 \cap B_5 , \ \ &
P_8 = B_1 \cap B_3 \cap B_6 , \\
P_9 = B_1 \cap B_3 \cap B_7 , \ \ &
P_{10} = B_1 \cap B_4 \cap B_5 , \ \ &
P_{11} = B_1 \cap B_4 \cap B_6 , \ \ &
P_{12} = B_1 \cap B_4 \cap B_7 , \\
P_{13} = B_1 \cap B_5 \cap B_6 , \ \ &
P_{14} = B_1 \cap B_5 \cap B_7 , \ \ &
P_{15} = B_1 \cap B_6 \cap B_7 , \ \ &
P_{16} = B_2 \cap B_3 \cap B_4 , \\
P_{17} = B_2 \cap B_3 \cap B_5 , \ \ &
P_{18} = B_2 \cap B_3 \cap B_6 , \ \ &
P_{19} = B_2 \cap B_3 \cap B_7 , \ \ &
P_{20} = B_2 \cap B_4 \cap B_5 , \\
P_{21} = B_2 \cap B_4 \cap B_6 , \ \ &
P_{22} = B_2 \cap B_4 \cap B_7 , \ \ &
P_{23} = B_2 \cap B_5 \cap B_6 , \ \ &
P_{24} = B_2 \cap B_5 \cap B_7 , \\
P_{25} = B_2 \cap B_6 \cap B_7 , \ \ &
P_{26} = B_3 \cap B_4 \cap B_5 , \ \ &
P_{27} = B_3 \cap B_4 \cap B_6 , \ \ &
P_{28} = B_3 \cap B_4 \cap B_7 , \\
P_{29} = B_3 \cap B_5 \cap B_6 , \ \ &
P_{30} = B_3 \cap B_5 \cap B_7 , \ \ &
P_{31} = B_3 \cap B_6 \cap B_7 , \ \ &
P_{32} = B_4 \cap B_5 \cap B_6 , \\
P_{33} = B_4 \cap B_5 \cap B_7 , \ \ &
P_{34} = B_4 \cap B_6 \cap B_7 , \ \ &
P_{35} = B_5 \cap B_6 \cap B_7  = \emptyset.
\end{array}
\]
In this way each nonempty $P_i$ for $i\in \{1,2,3,\ldots,34\}$ is contained in some
intersection of three of the sets $B_1,\ldots,B_7$ where each of these intersections
$P_j$ has cardinality at most $3$. Note that by insisting that we want an accuracy of
$a = 5$ we actually obtain an accuracy of $a = 3 < 5$ using the minimum number $k = 7$
of sets to form the intersection. As we saw, using $k = 6$ will only guarantee an accuracy
of $a = 6 > 5$. Hence, insisting on the accuracy of $a = 5$ is then equivalent to insisting on
an accuracy of $a = 3,4$ or $5$ which is achieved using $k = 7$ subsets of $[101]$.
\end{example}

\section{A formula for $s(n)$ for large $n$} 
\label{sec:formula}

In this section we will use a variation of the classical Stirling approximation formula 
and the Lambert $W$ function to obtain a closed formula for a function which is equal to either
$s(n)$ or $s(n) - 1$ for all $n$ that are large enough.

Recall that the Lambert $W$ function~\cite{Lambert-wiki}, \cite{Lambert-Wolfram} is the multivalued 
inverse of the function $w \mapsto we^w$ in the complex plane. When restricting to the real number 
line, there are two branches of the Lambert $W$ function: the principal branch 
$W_0 : [-1/e, \infty[ \longrightarrow [-1,\infty[$ and the second branch 
$W_{-1} : [-1/e,0[ \longrightarrow ]-\infty,-1]$. Both branches are analytic and bijective with $W_0$ 
strictly increasing and $W_{-1}$ strictly decreasing. The second branch $W_{-1}$ of the 
Lambert $W$ function has been been used significantly less than the main branch $W_0$, 
see Figure~\ref{fig:2nd-branch}.
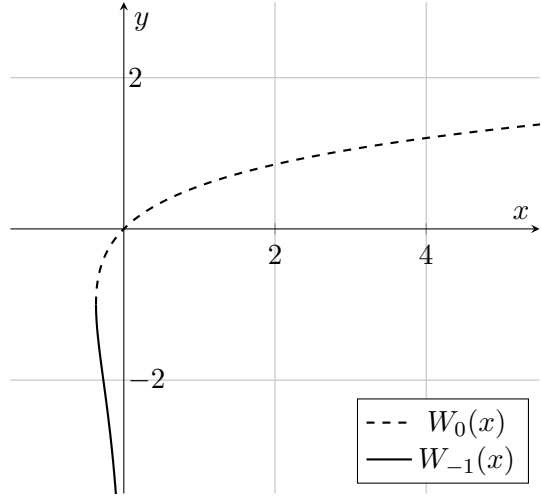
\begin{figure}[h!]
\begin{center}  
\begin{tikzpicture}
\begin{axis}[
    x=1cm, y=1cm,
    xtick distance=2,
    ytick distance=2, 
    ytick pos=right,
    xticklabel pos=right,
    xmin=-1.5, xmax=5.5,
    ymin=-3.5, ymax=3,
    axis lines=middle,
        grid=both,
    xlabel={$x$}, ylabel={$y$},
    clip=true,
    legend style={
        at={(0.98,0.02)},
        anchor=south east,
    },
]

\addplot[
    thick,
    dashed,
    domain=-1:2,
    samples=400,
    variable=\t,
]
({\t*exp(\t)},{\t});
\addlegendentry{$W_0(x)$}

\addplot[
    thick,
    solid,
    domain=-4:-1,
    samples=400,
    variable=\t,
]
({\t*exp(\t)},{\t});
\addlegendentry{$W_{-1}(x)$}

\end{axis}
\end{tikzpicture}
\end{center}
\caption{The two branches of the Lambert $W$ function when restricted to the reals.}
    \label{fig:2nd-branch}.
\end{figure}
Consider the function $y = y(x) = \frac{e^x}{x}$ for real $x\geq 1$. Here 
$y : [1,\infty[ \longrightarrow [e,\infty[$ is analytic, strictly increasing and bijective. 
Letting $z := -x$ we obtain $y = \frac{e^{-z}}{-z}$ and so $-y^{-1} = ze^z$. Since 
$z\in ]-\infty, -1]$ we have $-y^{-1}\in [-1/e,0[$ and hence, by the definition of the
second branch of the Lambert $W$ function, we have $z = W_{-1}(-y^{-1})$ and therefore 
the following lemma.
\begin{lemma}
\label{lmm:Lambert-1}
 The inverse of the function $y : [1,\infty[ \longrightarrow [e,\infty[$ presented by $y(x) = \frac{e^x}{x}$ 
 is given by $x = -W_{-1}(-y^{-1})$ where $W_{-1} : [-1/e,0[ \longrightarrow ]-\infty,-1]$ is the second
 branch of the Lambert $W$ function when restricted to the real number line.
\end{lemma}
The following variation, and a slight tightening, of the classical Stirling's approximation formula is due to 
Robbins~\cite{Stirling-Robbins} and is valid for all integers $n\geq 1$:
\[
\sqrt{2\pi n}\left(\frac{n}{e}\right)^ne^{\frac{1}{12n+1}}
<n! <
\sqrt{2\pi n}\left(\frac{n}{e}\right)^ne^{\frac{1}{12n}}.
\]
As a direct consequence, we obtain that for all integers $n\geq 1$ we have 
\begin{equation}
\label{eqn:Robbins}
n! = \sqrt{2\pi n}\left(\frac{n}{e}\right)^ne^{\theta_n},
\end{equation}
where $\frac{1}{12n+1} < \theta_n < \frac{1}{12n}$.
By (\ref{eqn:Robbins}) we then obtain that 
\begin{equation}
\label{eqn:ab}
\binom{n}{\lfloor n/2\rfloor}  = \frac{n!}{\lfloor n/2\rfloor!\lceil n/2\rceil!}
= \sqrt{\frac{n}{2\pi ab}}\cdot\frac{n^n}{a^ab^b}\cdot\phi_n
\end{equation}
where $a = \lfloor n/2\rfloor$ and $b = \lceil n/2\rceil$, and where 
$0< \phi_n < 1$ and $\phi_n\rightarrow 1$ as $n$ tends to infinity.
Since the function 
$c(x) = (x + 1/2)\log x$ is concave up for $x\geq 1$ we have in particular that
$\frac{c(x) + c(y)}{2} > c\left(\frac{x+y}{2}\right)$ for distinct $s$ and $y$
and hence 
\begin{equation}
\label{eqn:xy}
\sqrt{xy}x^xy^y > \left(\frac{x+y}{2}\right)^{x+y+1}.
\end{equation}
Since $a + b = n$ we obtain from (\ref{eqn:ab}) and (\ref{eqn:xy}) that 
$\binom{n}{\lfloor n/2\rfloor} < \sqrt{\frac{2}{\pi}}\frac{2^n}{\sqrt{n}}$ for all
$n\geq 1$.

Note that for even $n$ we obtain from (\ref{eqn:Robbins}) that 
\begin{equation}
\label{eqn:even}
\binom{n}{\lfloor n/2\rfloor} = \sqrt{\frac{2}{\pi}}\cdot\frac{2^n}{\sqrt{n}}\cdot\phi_n.
\end{equation}
For $n = 2k+1$ odd we similarly obtain from (\ref{eqn:Robbins}) that
\begin{equation}
\label{eqn:odd}
\binom{n}{\lfloor n/2\rfloor} 
= \sqrt{\frac{2}{\pi}}\cdot\frac{2^n}{\sqrt{n}}\cdot\phi_n
\cdot\frac{k+1/2}{\sqrt{k(k+1)}}\left(1 + \frac{1}{2k}\right)^k\left(1 - \frac{1}{2(k+1)}\right)^{k+1}.
\end{equation}
Since $(1 + \frac{r}{k})^k\rightarrow e^{r}$ as $k$ tends to infinity we see from (\ref{eqn:odd})
and (\ref{eqn:even}) that 
$\binom{n}{\lfloor n/2\rfloor} \sim \sqrt{\frac{2}{\pi}}\cdot\frac{2^n}{\sqrt{n}}$ for all 
$n$, even and odd. We summarize this in the following.
\begin{observation}
\label{obs:sim-less}
As $n$ tends to infinity we have that 
\[
\binom{n}{\lfloor n/2\rfloor} \sim \sqrt{\frac{2}{\pi}}\frac{2^n}{\sqrt{n}}.
\]
Further, for all $n\geq 1$ we have that 
\[
\binom{n}{\lfloor n/2\rfloor} < \sqrt{\frac{2}{\pi}}\frac{2^n}{\sqrt{n}}.
\]
\end{observation}
By Observation~\ref{obs:sim-less} we have for any $\epsilon > 0$ that 
\begin{equation}
\label{eqn:between}
\left(\sqrt{\frac{2}{\pi}} - \epsilon\right)\frac{2^n}{\sqrt{n}} < 
\binom{n}{\lfloor n/2\rfloor} < \sqrt{\frac{2}{\pi}}\frac{2^n}{\sqrt{n}}
\end{equation}
for all $n\geq N_{\epsilon}$ where $N_{\epsilon}\in{\nats}$ is some fixed natural number
depending only on $\epsilon$.

For a real $\alpha > 0$ let $g_{\alpha}(x) = \alpha\frac{2^x}{\sqrt{x}}$ for all real $x\geq 1$.
The function $g_{\alpha} : [1,\infty[ \longrightarrow [2\alpha, \infty[$ is clearly 
strictly increasing and bijective. 
By Lemma~\ref{lmm:Lambert-1} the inverse $g^{-1}_{\alpha}$ is given by 
\begin{equation}
\label{eqn:g-alpha-1}
g^{-1}_{\alpha}(x) = \frac{-1}{2\log2}W_{-1}\left(\frac{-2\alpha^2\log2}{x^2}\right).
\end{equation}
If $f : {\reals}_{+}\rightarrow {\reals}_{+}$ is an 
analytic, bijective and increasing function with $f(n) = \binom{n}{\lfloor n/2\rfloor}$ for each 
natural number $n$, then $s(n)$ from Definition~\ref{def:s(n)} is given by 
$s(n) = \lceil f^{-1}(n)\rceil$. By (\ref{eqn:between}) and (\ref{eqn:g-alpha-1}) we obtain 
the following.
\begin{theorem}
\label{thm:s-between}
For each $\epsilon > 0$ there is an integer $N_{\epsilon}$ such that 
\[
\left\lceil\frac{-1}{2\log2}W_{-1}\left(\frac{-4\log2}{\pi n^2}\right)\right\rceil
\leq s(n) \leq 
\left\lceil\frac{-1}{2\log2}W_{-1}\left(\frac{-2\alpha^2\log2}{n^2}\right)\right\rceil
\]
for all $n\geq N_{\epsilon}$, where $\alpha = \sqrt{\frac{2}{\pi}} - \epsilon$ and $W_{-1}$ is the second 
branch of the Lambert $W$ function.
\end{theorem}

The asymptotic behavior of the Lambert $W$ function, both its branches $W_0$ and $W_{-1}$,
are well-known and have been studied~\cite{Lambert-wiki,Lambert-Wolfram} and~\cite{Corless-and-co}.
For the second branch $W_{-1}$ we have that
\begin{equation}
    \label{eqn:Lambert-asymp}
W_{-1}(t) = \log(-t) - \log(-\log(-t)) + o(1)    
\end{equation}
for $t\in [-1/e,0[$ where $o(1)$ tends to zero as $t$ tends to zero from below. 
Consequently, for large $x$ we then obtain from (\ref{eqn:Lambert-asymp}) and 
(\ref{eqn:g-alpha-1}) that
\begin{eqnarray*}
g^{-1}_{\alpha}(x) & = & \frac{-1}{2\log2}W_{-1}\left(\frac{-2\alpha^2\log2}{x^2}\right) \\
& = & \frac{-1}{2\log2}\left(\log\left(\frac{2\alpha^2\log2}{x^2}\right) 
-\log\left(-\log\left(\frac{2\alpha^2\log2}{x^2}\right)\right) + o(1) \right) \\
& = & \lg x - \frac{1}{2}\lg(2\alpha^2\log2) + \frac{1}{2}\lg(2\log x - \log(2\alpha^2\log2)) + o(1) \\
& = & \lg x - \frac{1}{2}\lg(2\alpha^2\log2) + \frac{1}{2}\lg(2\log x) + o(1) \\
& = & \lg x + \frac{1}{2}\lg(\lg x) - \lg\alpha + o(1),
\end{eqnarray*}
where $o(1)$ is a function of $x$ that tends to zero when $x$ tends to infinity. 
So, we specifically have that 
\begin{equation}
\label{eqn:explicit-a}    
g^{-1}_{\alpha}(x) = \lg x + \frac{1}{2}\lg(\lg x) - \lg\alpha + r_{\alpha}(x)
\end{equation}
where $r_{\alpha}(x) \rightarrow 0$ as $x$ tends to infinity. 

Let $\epsilon > 0$ be a given. Since $t\mapsto\lg t$ is continuous and strictly 
increasing there is a ${\epsilon}' > 0$ such that 
$0 < -\lg\alpha - \frac{1}{2}\lg(\pi/2) < \epsilon/2$ where 
$\alpha = \sqrt{\frac{2}{\pi}} - {\epsilon}'$. Let $N_{\epsilon}$ be such that  
$|r_{\alpha}(n) - r_{\sqrt{2/\pi}}(n)| < -\lg\alpha - \frac{1}{2}\lg(\pi/2)$ for all $n\geq N_{\epsilon}$.
A reformulation of Theorem~\ref{thm:s-between} is then 
$\left\lceil g^{-1}_{\sqrt{2/\pi}}(n)\right\rceil \leq s(n) \leq \left\lceil g^{-1}_{\alpha}(n)\right\rceil$ 
for all $n \geq N_{\epsilon}$. By (\ref{eqn:explicit-a}) we have
\begin{eqnarray*}    
g^{-1}_{\alpha}(n) - g^{-1}_{\sqrt{2/\pi}}(n) & = &
\left(-\lg\alpha + r_{\alpha}(n)\right) - 
  \left(-\lg\left(\sqrt{2/\pi}\right) + r_{\sqrt{2/\pi}}(n)\right) \\
& = &  \left(-\lg\alpha - \frac{1}{2}\lg(\pi/2)\right) + \left(r_{\alpha}(n) - r_{\sqrt{2/\pi}}(n)\right) \\
& \in & [0,\epsilon[.
\end{eqnarray*}
Therefore we have the following.
\begin{corollary}
\label{cor:diff-epsilon}
For every $\epsilon > 0$ there is an $N_{\epsilon}\in {\nats}$ and a discrepancy function
$\delta : \nats \rightarrow \reals_+$ with $0\leq \delta(n) < \epsilon$ such that 
\[
s(n) = \left\lceil\lg n + \frac{1}{2}\lg(\lg n) + \frac{1}{2}\lg(\pi/2) + \delta(n)\right\rceil
\]
for all $n\geq N_{\epsilon}$. In particular, $\sigma(n)\leq s(n)\leq \sigma(n)+1$ for all $n\geq N_{\epsilon}$ 
where 
\[
\sigma(n) = \left\lceil\lg n + \frac{1}{2}\lg(\lg n) + \frac{1}{2}\lg(\pi/2)\right\rceil. 
\]
\end{corollary}
Note that the last displayed form of $\sigma(n)$ yields a more accurate description 
of the asymptotic behavior of $s(n)$ than what 
is mentioned in A305233~\cite{A305233}.

By Theorem~\ref{thm:intrep} we can 
restate the last sentence of Corollary~\ref{cor:diff-epsilon} in terms of representing sets.
\begin{corollary}
\label{cor:rep-nice}
There is an $N\in\nats$ such that for every $n\geq N$ the set $[n]$ can be represented by
at most $\left\lceil\lg n + \frac{1}{2}\lg(\lg n) + \frac{1}{2}\lg(\pi/2)\right\rceil + 1$
sets from $2^{[n]}$. This upper bound is the best possible up to a difference of $1$.
\end{corollary}
\begin{rmk*}
If one is only concerned about the leading $\lg n$ term for the asymptotic behavior of $s(n)$ 
from Corollary~\ref{cor:diff-epsilon}, then a slightly quicker argument similar to Lemma 3.26 from~\cite{A-A-Dawkins} can be used
to show that if $f,g : {\reals}_{+} \rightarrow {\reals}_{+}$ be increasing and bijective functions such that
(a) $f(x) < g(x)$ for all sufficiently large $x$, (b) $f(x) \sim g(x)$ as $x\rightarrow \infty$, 
and (c) $g(x)/x$ is an increasing function for all sufficiently large $x$, then 
$f^{-1}(x)\sim g^{-1}(x)$ as $x\rightarrow\infty$. 
Since $e^x/x^k$ is an increasing function for all $x$ large enough we have that for 
$g(x) = \sqrt{\frac{2}{\pi}}\frac{2^x}{\sqrt{x}}$ the function $\frac{g(x)}{x}$ is an increasing
real function for all $x$ large enough. By Observation~\ref{obs:sim-less} we would then
be able to obtain $s(n) = \lceil f^{-1}(n)\rceil \sim \lg n$ as $n$ tends to infinity.
This is however weaker than the above Corollary~\ref{cor:diff-epsilon}.
\end{rmk*}

\section{Summary}
\label{sec:summary}

We briefly discuss some of the main results in this article. The novel
contributions in this paper are from Sections~\ref{sec:representing}, \ref{sec:set-partitions}
and \ref{sec:formula}. 

In Section~\ref{sec:representing} the main result is
Theorem~\ref{thm:intrep} that states that the minimum number of sets from $2^{[n]}$
that are needed to represent $[n]$ is exactly $s(n)$ from Definition~\ref{def:s(n)}; 
the smallest cardinality of a set that contains $n$ subsets where no set contains another. 

In Section~\ref{sec:set-partitions} the main result is Theorem~\ref{thm:seta}
which generalizes Theorem~\ref{thm:intrep} and describes how accurately one can
pinpoint an element in $[n]$ when we have $s(n)$ or fewer sets
from $2^{[n]}$. 

In Section~\ref{sec:formula} we obtain a tight asymptotic formula for $s(n)$ 
for large $n$ in Corollary~\ref{cor:diff-epsilon}, which is based on Theorem~\ref{thm:s-between}. 
By "tight" we mean that for large enough $n$ the given lower
bound and upper bound differ by at most $1$. Casually speaking, we see from 
Corollaries~\ref{cor:diff-epsilon} and \ref{cor:rep-nice} 
that the minimum number $\rho(n) = s(n)$ of labels
needed to identify each element in $[n]$ is then at most $\ceil{\lg n +\frac{1}{2}\lg(\lg n)} + 2$
which is considerably better than the number of roughly $2\lg n$ needed when applying 
classical binary search, as shown in Example~\ref{exa:kcube}. It is finally worth
emphasizing that Corollary~\ref{cor:diff-epsilon} is the best possible such lower
bound for the number of labels.

\bibliographystyle{amsalpha}
\bibliography{Geirbib}

\end{document}